\documentclass{article}

\usepackage{amsmath,amssymb,amsthm}
\usepackage{enumitem}
\usepackage{fixltx2e}
\usepackage{mathtools}
\usepackage{tikz}
\usepackage{tikz-cd}
\usepackage[textsize=small]{todonotes}
\usepackage{xparse}
\usepackage{xspace}
\usepackage{mathtools}
\usepackage{dsfont}
\usepackage[all]{xy}
\usepackage{color}
\usepackage{verbatim}
\usepackage{hyperref}
\usepackage{mathrsfs}
\usepackage{microtype}

\numberwithin{equation}{section}
\mathtoolsset{showonlyrefs=true}

\let\blb\mathbb

\def \PP{{\blb P}}

\def \ZZ{{\blb Z}}

\def \NN{{\blb N}}

\def\Gal{\operatorname {Gal}}

\DeclareMathOperator{\Br}{Br}

\theoremstyle{definition}

\newtheorem{lemma}{Lemma}[section]

\newtheorem{theorem}[lemma]{Theorem}

\newtheorem{conjecture}[lemma]{Conjecture}

\newtheorem{remark}[lemma]{Remark}

\DeclareMathOperator\Hom{Hom}

\def\opp{\operatorname{op}}

\def\ind{\operatorname{ind}}

\usepackage{xcolor}
\hypersetup{
    colorlinks,
    linkcolor={red!50!black},
    citecolor={blue!50!black},
    urlcolor={blue!80!black}
}

\mathchardef\mhyphen="2D

\newcounter{todocounter}
\DeclareDocumentCommand\addreference{g}{\stepcounter{todocounter}\todo[color = blue!30, fancyline]{\thetodocounter. Add reference\IfNoValueF{#1}{: #1}}\xspace}
\DeclareDocumentCommand\checkthis{g}{\stepcounter{todocounter}\todo[color = red!50, fancyline]{\thetodocounter. Check this\IfNoValueF{#1}{: #1}}\xspace}
\DeclareDocumentCommand\fixthis{g}{\stepcounter{todocounter}\todo[color = orange!50, fancyline]{\thetodocounter. Fix this\IfNoValueF{#1}{: #1}}\xspace}
\DeclareDocumentCommand\expand{g}{\stepcounter{todocounter}\todo[color = green!50, fancyline]{\thetodocounter. Expand\IfNoValueF{#1}{: #1}}\xspace}

\title{Non-split Brauer-Severi varieties do not admit full exceptional collections}
\author{Theo Raedschelders}
\date{ }
\begin{document}

\maketitle

\begin{abstract}
Recently, Novakovi\'c conjectured that non-split Brauer-Severi varieties do not admit full strong exceptional collections. In this short note, we explain how a stronger version of this conjecture follows easily from known results on noncommutative motives.
\end{abstract}

\newcommand\blfootnote[1]{%
  \begingroup
  \renewcommand\thefootnote{}\footnote{#1}%
  \addtocounter{footnote}{-1}%
  \endgroup
}

\section{Introduction}

For an arbitrary field $k$, Novakovi\'c stated the following as a conjecture in~\cite{novakovic2015tilting}:
	\begin{conjecture}
	\label{novikov}
	Let $X \neq \PP_k^n$ be a $n$-dimensional Brauer-Severi variety. Then $D^b(X)$ does not admit a full strongly exceptional collection.
	\end{conjecture}
He proves the conjecture in dimension $n \leq 3$~\cite{novakovic2016no} by exploiting the transitivity of the braid group action on full exceptional collections for $\PP^n_k$ to reduce to an equivalence $D^b(A) \cong D^b(k)$. If $A \cong M_l(D)$, for $D$ a division algebra over $k$, these are just the categories of $\ZZ$-graded vector spaces over $D$, respectively $k$, so there is an equivalence only if $D$ is isomorphic to $k$. Since the transitivity of the braid group action (which is only established for $n \leq 3$) is only used to be able to reduce to a single semi-orthogonal component, this suggests that noncommutative motives might provide the right framework for this conjecture. Using some results from~\cite{tabuada2014noncommutative} on noncommutative motives of separable algebras, we prove a slightly stronger version of Conjecture~\ref{novikov}, showing that non-split Brauer-Severi varieties do not admit full \'etale exceptional collections. 

\section{Noncommutative motives of separable algebras}

To any small dg-category $\mathcal{A}$, one can associate (functorially) its noncommutative motive $U(\mathcal{A})$, which takes values in a category ${\tt Hmo}_0(k)$. This category has as objects small dg-categories, and for two such categories $\mathcal{A}$ and $\mathcal{B}$,
	$$
	\Hom_{{\tt Hmo}_0(k)}(\mathcal{A},\mathcal{B}) \cong K_0 {\tt rep}(\mathcal{A},\mathcal{B}),
	$$
where ${\tt rep}(\mathcal{A},\mathcal{B})$ is the full triangulated subcategory of $D(\mathcal{A}^{\opp} \otimes^{\mathbb{L}} \mathcal{B})$ consisting of those $\mathcal{A}$-$\mathcal{B}$-bimodules $B$ such that for every $x \in \mathcal{A}$, the right $\mathcal{B}$-module $B(x,-)$ is a compact object in $D(\mathcal{B})$. The composition is induced by the derived tensor product of bimodules.

More details on the construction of $U$ can be found in~\cite{tabuada2005additive}, but for the purposes of this note, we will only need that $U$ is a ``universal additive invariant''. An additive invariant is any functor $E:{\tt dgcat}(k) \to D$ taking values in an additive category $D$ such that:
\begin{enumerate}
\item it sends dg-Morita equivalences to isomorphisms,
\item for any pre-triangulated dg-category $\mathcal{A}$, with full pre-triangulated dg-subcategories $\mathcal{B}$ and $\mathcal{C}$ giving rise to a semi-orthogonal decomposition
$$
{\tt H}^0(\mathcal{A})=\langle {\tt H}^0(\mathcal{B}), {\tt H}^0(\mathcal{C}) \rangle,
$$
the morphism $E(\mathcal{B}) \oplus E(\mathcal{C}) \to E(\mathcal{A})$ induced by the inclusions is an isomorphism.
\end{enumerate}

We now review some results from~\cite{tabuada2014noncommutative}. Remember that the category of noncommutative Chow motives $\text{NChow}(k)$ is defined as the idempotent completion of the full subcategory of ${\tt Hmo}_0(k)$ containing the smooth and proper dg-categories. Now let $\text{Sep}(k)$ (respectively $\text{CSep}(k)$) denote the full subcategory of $\text{NChow}(k)$ consisting of the $U(A)$, for $A$ a separable (respectively commutative separable) $k$-algebra. Also let $\text{CSA}(k)^{\oplus}$ denote the closure under finite direct sums of the full subcategory of $\text{NChow}(k)$ consisting of the $U(A)$, for $A$ a central simple $k$-algebras. Note that the $\oplus$ is there since central simple $k$-algebras are not closed under products, whereas (commutative) separable algebras are. In this way $\text{Sep}(k), \text{CSep}(k)$ and $\text{CSA}(k)^{\oplus}$ are additive symmetric monoidal categories. 

	\begin{theorem}\cite[Corollary 2.13]{tabuada2014noncommutative}
	\label{tabvdb2}
	There is an equivalence of categories
	$$
	\{U(k)^{\oplus n} \vert n \in \NN\}\simeq\text{CSA}(k)^{\oplus} \times_{\text{Sep}(k)} 
	\text{CSep}(k),
	$$
	i.e. $\{U(k)^{\oplus n} \vert n \in \NN\}$ is a $2$-pullback of categories with respect to the obvious inclusion morphisms.
	\end{theorem}
	
For a central simple algebra $A$ over $k$, denote by $\ind(A)$ and $\deg(A)$ the index (respectively degree) of $A$. Then by~\cite[Proposition 4.5.16]{MR2266528}, $A$ admits a $p$-primary decomposition 
	$$
	A=\bigotimes_{i=1}^k A^{p_i},
	$$
where $A^{p_i}$ is uniquely characterised by the property $\ind(A^{p_i})=p_i^{n_i}$ if
	$$
	\ind(A)=p_1^{n_1} \cdots p_k^{n_k}
	$$
is the primary decomposition. 

	\begin{theorem}\cite[Theorem 2.19]{tabuada2014noncommutative}
	\label{tabvdb}
	Given central simple $k$-algebras $A_1, \ldots, A_n$ and $B_1, \ldots, B_m$, the 	
	following two conditions are equivalent:
		\begin{enumerate}
		\item There is an isomorphism of noncommutative motives:
			$$
			U(A_1) \oplus \cdots \oplus U(A_n) \simeq U(B_1) \oplus \cdots \oplus U(B_m).
			$$
		\item The equality $n=m$ holds, and for all $1 \leq j \leq n$ and all $p$
			$$
			[B_j^p]=[A_{\sigma_p(j)}^p]
			$$
			holds in $\Br(k)$, for some permutations $\sigma_p$ depending on $p$.
		\end{enumerate}
	\end{theorem}

\begin{remark}
Though the isomorphism classes of objects in $\text{CSA}(k)^{\oplus}$ are in some sense understood by Theorem~\ref{tabvdb}, this is not true for $\text{CSep}(k)$. In fact, using the (additive) equivalence $\text{CSep}(k) \simeq \text{Perm}(G)$, where $G=\Gal(k_{\text{sep}}/k)$, and $\text{Perm}(G)$ is the category of permutation $G$-modules, interesting examples can be obtained from integral representation theory, see~\cite[Remark 2.5, 2.6]{tabuada2014noncommutative}.
\end{remark}

\section{Brauer-Severi varieties and full \'etale exceptional collections}

Denote by $BS(A)$ the Brauer-Severi variety associated to a central simple $k$-algebra $A$. We will say (see also~\cite{Orlov:2014aa}) that an object $E \in D^b(BS(A))$ satisfying $\Hom(E,E[i])=0$ for all $i \neq 0$ is 
	\begin{itemize}
	\item semi-exceptional if $\Hom(E,E)=S$ is a semisimple $k$-algebra,
	\item \'etale exceptional if $\Hom(E,E)=L$ is an \'etale $k$-algebra.
	\end{itemize}
It is well known~\cite{bernardara2009semiorthogonal} that $BS(A)$ has a full semi-exceptional collection giving rise to a semi-orthogonal decomposition
	\begin{equation}
	\label{SOD}
	D^b(BS(A)) = \langle D^b(k), D^b(A), \ldots, D^b(A^{\otimes \deg(A)-1}) \rangle.
	\end{equation}

The following theorem now provides a positive answer to Conjecture~\ref{novikov}.

\begin{theorem}
\label{theorem:nofullexc}
Non-split Severi-Brauer varieties do not admit full \'etale exceptional collections.
\end{theorem}
\begin{proof}
Suppose $A$ is non-split and $\deg(A)=d$. Then if $BS(A)$ has a full \'etale exceptional collection, we deduce from~\eqref{SOD} and additivity of $U(-)$ with respect to semi-orthogonal decompositions that there is an isomorphism
$$
U(k) \oplus U(A) \oplus \cdots \oplus U(A^{\otimes d-1}) \simeq U(D^b(BS(A))) \cong U(L_1) \oplus \cdots \oplus U(L_d),
$$
where the $L_i$ are \'etale $k$-algebras. Using 
Theorem~\ref{tabvdb2} and the universal property of fibre products, this isomorphism 
gives rise to an isomorphism
$$
U(k) \oplus U(A) \oplus \cdots \oplus U(A^{\otimes d-1}) \simeq U(k)^{\oplus d}.
$$
Now by Theorem~\ref{tabvdb}, for all $p: [A^p]=[k]$ in $\Br(k)$, so $[A]=[k]$ or in other 
words $A$ should split.
\end{proof}
	
\begin{remark}
This result formalizes (in this case) the intuition that for varieties defined over arbitrary fields, one should consider semi-exceptional collections instead of usual exceptional collections.
\end{remark}

\providecommand{\bysame}{\leavevmode\hbox to3em{\hrulefill}\thinspace}
\providecommand{\MR}{\relax\ifhmode\unskip\space\fi MR }
\providecommand{\MRhref}[2]{%
  \href{http://www.ams.org/mathscinet-getitem?mr=#1}{#2}
}
\providecommand{\href}[2]{#2}

\bibliographystyle{amsplain}
\def\cprime{$'$} \def\cprime{$'$} \def\cprime{$'$}
\providecommand{\bysame}{\leavevmode\hbox to3em{\hrulefill}\thinspace}
\providecommand{\MR}{\relax\ifhmode\unskip\space\fi MR }
\providecommand{\MRhref}[2]{%
  \href{http://www.ams.org/mathscinet-getitem?mr=#1}{#2}
}
\providecommand{\href}[2]{#2}

\end{document}